\documentclass[10pt,a4paper, reqno]{amsart}
\usepackage{mathrsfs}

\usepackage{amsmath,amsfonts,verbatim}
\usepackage{latexsym}
\usepackage{amssymb,leftidx}
\usepackage{extarrows}
\usepackage{overpic}
\usepackage{color}
\usepackage{epsfig}
\usepackage{subfigure}
\usepackage{tikz}
\usepackage{bbm}
\usepackage{tikz, caption}
\usepackage[font=small,labelfont=bf]{caption}

\usepackage[colorlinks=true,backref=page]{hyperref}
\hypersetup{urlcolor=blue, citecolor=blue}

\usepackage{amssymb}
\usepackage{mathrsfs}
\usepackage{amscd}
\usepackage{cite}

\newcommand\R{\mathbb{R}}

\usepackage{graphicx}
\usepackage{float}

\numberwithin{equation}{section}
\newtheorem{proposition}{Proposition}[section]
\newtheorem{definition}{Definition}[section]
\newtheorem{lemma}{Lemma}[section]
\newtheorem{theorem}{Theorem}[section]

\allowdisplaybreaks
\begin{document}
\title[$L^p$-estimates for partial inverse wave]{$L^p$-estimates for the wave equation with partial inverse-square potentials}

\author{Jialu Wang}
\address{School of Mathematics and Statistics, Zhengzhou University, Zhengzhou, Henan, China, 450001;} \email{wang\_jialu6@163.com}

\author{Chengbin Xu}
\address{School of Mathematics and Statistics, Qinghai Normal University, Xining, Qinghai, China, 810008; }
\email{xcbsph@163.com}

\author{Fang Zhang}
\address{Department of Mathematics, Beijing Institute of Technology, Beijing, China, 100081;} \email{zhangfang@bit.edu.cn}

\author{Junyong Zhang}
\address{Department of Mathematics, Beijing Institute of Technology, Beijing 100081}
\email{zhang\_junyong@bit.edu.cn}

\begin{abstract}
This paper investigates $L^p$-estimates for solutions to the wave equation perturbed by a scaling-critical partial inverse-square potential. We study a model in which the singularity of the potential appears only in a subset of the variables, corresponding to the Schr\"{o}dinger operator $\mathcal{H}_a = -\Delta_x - \Delta_y + a/|x|^2$ on $\mathbb{R}^{2+n}$. Using spectral analysis, we establish the $L^p$-boundedness of the wave propagator $(1+\sqrt{\mathcal{H}_a})^{-\gamma} e^{it\sqrt{\mathcal{H}_a}}$ for a range of exponents $\gamma$ and $p$ satisfying $|1/p -1/2| < \gamma/(n+1)$. The key ingredients are the spectral measure kernel of the partial inverse-square  operator $\mathcal{H}_a$ and the complex interpolation argument.
\end{abstract}
 \maketitle

\begin{center}
\begin{minipage}{120mm}
\end{minipage}
\end{center}

\section{Introduction}
This paper studies $L^p$-estimates for the solution to the wave equation associated with the self-adjoint Schr\"odinger operator $$H_V=-\Delta+V(x)$$ on $L^2(\R^d)$  with $d\geq 2$ where the operator  $\Delta$ is the usual Laplacian on $\R^d$ and the potential $V(x)$  is a real-valued function.
More precisely, we consider the following Cauchy problem for the wave equation
\begin{equation}\label{eq:wave}
\begin{cases}
\partial_{tt}u + H_{V}u = 0, & (t,x) \in \mathbb{R} \times \mathbb{R}^{d}, \\
u(0,x) = f(x), \quad \partial_{t}u(0,x) = g(x), & x \in \mathbb{R}^{d}.
\end{cases}
\end{equation}
The study of $L^p$-estimates for solutions to wave equations constitutes a well-established and active research program in harmonic analysis and partial differential equations. Its development dates back to the seminal works of Peral \cite{P} and Miyachi \cite{Mi1,Mi2,Mi3}, who determined the sharp range of $L^p$-boundedness for the wave propagator $\frac{\sin(t\sqrt{-\Delta})}{\sqrt{-\Delta}}$. \vspace{0.2cm}

When considering potential perturbations, the complexity of the problem increases significantly. Particularly, for potentials with scaling-critical properties, the inverse-square Hardy-type potential $V(x) = a|x|^{-2}$ stands out as a typical example that has attracted widespread attention, see \cite{BPSS, BPST, KMVZZ, MZZ, PST, VZ} and related references. However, studying wave equations associated with scaling-critical potential operators presents fundamental difficulties, as the spectral structure of the operator undergoes an essential change, rendering the Fourier transform inapplicable. Nevertheless, in the context of Strichartz and dispersive estimates, extensive work has been devoted to wave equations with electricmagnetic potentials. For decaying potentials, such as short-range potentials, perturbative methods are often applicable, allowing results from the free case to be extended to the perturbed setting. However, for long-range singular potentials like the inverse-square potential, perturbative approaches generally fail. Breakthrough contributions on the inverse-square operator were made by Burq et al. \cite{BPSS,BPST}, who established Strichartz estimates for both wave and Schr\"odinger equations with such critical potentials. This type of potential lies on the borderline of validity for dispersive estimates, as shown by the work of Duyckaerts \cite{D} and Goldberg-Vega-Visciglia \cite{GVV}.
Fanelli et al. \cite{FFFP, FZZ} systematically derived dispersive estimates for Schr\"odinger and wave equations with critical electromagnetic potential, while Gao et al. \cite{GWZZ, GYZZ} established uniform resolvent estimates, decay and Strichartz estimates in electromagnetic fields. \vspace{0.2cm}

In this paper, we consider the partial inverse-square Schr\"odinger operator $H_{V}=\mathcal{H}_{a}$ on $L^2(\R^{2+n})$
\begin{equation}\label{La}
\mathcal{H}_a = -\Delta_x - \Delta_y + \frac{a}{|x|^2}, \quad x \in \mathbb{R}^{2} \setminus \{0\}, \quad y \in \mathbb{R}^{n}, \quad a > 0,
\end{equation}
which arises naturally many-body quantum systems and differs from the classical inverse-square potential model. On the one hand, this partial inverse-square operator exhibits rich mathematical structure: the potential is more singular and the singularity is restricted to the $x$-variables, breaking full rotational symmetry and rendering many symmetry-based analytical techniques inapplicable. On the other hand, the physical significance of this operator also lies in its connection to singular two-body quantum systems: via a suitable coordinate transformation, the two-particle Hamiltonian
$$H_{V,2} = -\Delta_{x_1} - \Delta_{x_2} + a|x_1 - x_2|^{-2}
$$
can be mapped precisely to our operator $\mathcal{H}_a$ in $\mathbb{R}^4$ when $n=2$, which serves as a pivotal starting point for exploring wave propagation behaviors in many-body quantum systems involving singular interactions. Regarding this operator $\mathcal{H}_a$, the Strichartz estimates have been previously established in \cite{ZZ}, we now turn to investigate the $L^p$ estimates for the wave equation associated with the operator $\mathcal{H}_a$.\vspace{0.2cm}

Regarding other potentials, certain $L^p$-estimates have also been established. For instance, the case of the Hermite operator $-\Delta+|x|^2$ was addressed in \cite{NT}, and more general operators of the form $-\Delta+V$ were studied in \cite{Zhong}. Later, Jotsaroop and Thangavelu \cite{JT} showed the operator $\frac{\sin\sqrt{G}}{G} $, in which the Grushin operator $G=-\Delta-|x|^2\partial_t^2$, is bounded on $L^p(\R^{n+1})$ for every $p$ satisfying $|\frac{1}{p}-\frac{1}{2}|<\frac{1}{n+2}$, thereby achieving the $L^p$ estimate for the solution. Correspondingly, they proved $L^p \to L^2$ estimate for wave equation with the Grushin operator via employing the $L^p$ boundedness of Riesz transform in subsequent article \cite{TD}. The study of $L^p$-boundedness for operators with singular potentials can be traced back to the work of Cheeger and Taylor \cite{CT} on conical singularities, followed by contributions by Li and Lohou\'{e} \cite{L,LN} on conic manifolds. In particular, our recent work \cite{WZZZ} developed $L^p$-estimates for the 2D wave equation with scaling-critical magnetic potentials by combining pointwise estimates of the kernel of an analytic operator family with Stein's interpolation theory.\vspace{0.2cm}

However, the methodology established above cannot be directly transferred to the partial inverse-square operator $\mathcal{H}_a$, because the potential $V=a/|x|^{2}$ is radial only with respect to the first variable $x\in\mathbb{R}^2$ and independent of the last variable $y\in\mathbb{R}^n$ (no decaying in $y$-direction), thus the particular partial inverse-square operator $\mathcal{H}_a$ is more singular. Nevertheless, we explore the spectral measure derived from the distorted Fourier transform, which serves as an effective alternative to the classical Fourier transform method in the free case. Notably, in \cite{ZZ}, we investigated the dispersive behaviors of Schr\"odinger equations and wave equations with the operator $\mathcal{H}_a$ by establishing distorted Fourier transform and the spectral measure. Inspired by \cite{ZZ}, this paper aims to establish the $L^p$-regularity for the wave propagator $(1+\sqrt{\mathcal{H}_a})^{-\gamma} e^{it\sqrt{\mathcal{H}_a}}$.  Unlike Strichartz estimates, which primarily describe global behavior in mixed space-time norms, $L^p$ estimates reflect the propagation characteristics of spatial regularity. In contrast, the $L^p$ estimates of the partial inverse-square operator $\mathcal{H}_a$ requires us to proceed from microlocal analysis based on the spectral measure and its kernel representation. The key to studying the $L^p$ estimates of operator $\mathcal{H}_a$ lies in obtaining precise pointwise estimates of spectral measure, thereby reformulating the problem into the analysis of oscillatory integrals. \vspace{0.2cm}

We characterize the $L^p$ regularity of the solution to \eqref{eq:wave} by establishing the $L^p$-boundedness of an associated oscillatory integral operator. It is well-known that the solution can be expressed via the spectral calculus as:
\begin{equation}
u(t, \cdot) = \cos(t\sqrt{\mathcal{H}_{a}}) f(\cdot) + \frac{\sin(t\sqrt{\mathcal{H}_{a}})}{\sqrt{\mathcal{H}_{a}}} g(\cdot). \label{eq:solution_rep}
\end{equation}
Thus, taking the $L^p$-boundedness of operators of the form $(1+\sqrt{\mathcal{H}_{a}})^{-z} e^{it\sqrt{\mathcal{H}_{a}}}$ into account is crucial. Our main result is the following.

\begin{theorem} \label{thm:main}
Let $\mathcal{H}_{a}$ be in \eqref{La} and let $\gamma > 0$ and $1 \leq p \leq \infty$ satisfy $|\frac{1}{p} - \frac{1}{2}| < \frac{\gamma}{n+1}$. Then there exists a constant $C(p, \gamma) > 0$ such that for any $t > 0$ and all $f \in L^{p}(\mathbb{R}^{2+n}),$ we have
\begin{equation}
\|(1+\sqrt{\mathcal{H}_{a}})^{-\gamma} e^{it\sqrt{\mathcal{H}_{a}}} f\|_{L^{p}(\mathbb{R}^{2+n})} \leq C(p, \gamma)(1+t)^{\gamma} \|f\|_{L^{p}(\mathbb{R}^{2+n})}. \label{eq:main_estimate}
\end{equation}
\end{theorem}

In order to prove Theorem \ref{thm:main}, our proof strategy centers on the establishment of Proposition \ref{prop:tz_bounds}. The core of the argument lies in a detailed pointwise estimate of the integral kernel of analytic operator family $T_z$. Through a dual argument and complex interpolation, the problem reduces to establishing the $L^2$-boundedness of the operator family $T_{z} = (1 + \sqrt{\mathcal{H}_a})^{-z}e^{i\sqrt{\mathcal{H}_a}}$ at $\Re(z)=0$ and its $L^1$-boundedness at $\Re(z)=\frac{n+1}{2}+\varepsilon$. The $L^2$-boundedness follows directly from the unitarity of the operator, while the $L^1$ estimate constitutes the core difficulty of the proof. \vspace{0.2cm}

To prove the $L^1$ estimate, we conduct a more refined analysis through a detailed explicit expression of the spectral measure. Specifically, utilizing the known expression for the spectral measure established in \cite{ZZ}, we decompose the kernel function via a dyadic partition into low-frequency parts $(U,V)$ and a high-frequency part $G$. For the low-frequency part $V$, we utilize the Gaussian boundedness of the heat kernel associated with the operator $\mathcal{H}_a$. Also, we further decompose the high-frequency kernel $G$ into three components $G_1, G_2, G_3$, each of these terms originates from contributions corresponding to different geometric structures in the spectral measure. For the low-frequency part $U$ and each component of the high-frequency kernel $G$, we apply integration by parts to the relevant oscillatory integrals, ultimately proving their integrability and thereby establishing the crucial $L^1$-boundedness. In this process, the treatment of $G_2$ and $G_3$ presents a more significant technical challenge compared with the term $G_1$. The main difficulty stems from the phase function $\vec{n}_s^1$ and $\vec{n}_s^2$ appearing in the oscillatory integral of the spectral measure, which involves an angular correction induced by the singular potential. As a result, neither the classical method of stationary phase nor a simple scaling argument remains directly applicable. Fortunately, by a refined approach incorporating dyadic decomposition and suitable changes of variables, we are able to overcome this obstacle. \vspace{0.2cm}

The structure of this paper is as follows. In Section \ref{pre}, we review the necessary preliminaries on the spectral measure and heat kernel of $\mathcal{H}_a$, and introduces the analytic operators ${T_z}$. Section \ref{proof-prop} is devoted to the proof of Proposition \ref{prop:tz_bounds}, the cornerstone of Theorem \ref{thm:main}, which involves the detailed decomposition and estimation of the kernel. Section \ref{proof-thm:main} completes the proof of Theorem \ref{thm:main} via the complex interpolation theory and duality argument. Finally, the proof of the integrability for the singular integral arising from the high-frequency part of the kernel function is presented as a lemma in Appendix \ref{app}.

\section{Preliminaries}\label{pre}
This section collects the fundamental tools and known results that form the basis of our analysis. We begin by recalling the explicit expression for the spectral measure of the operator $\sqrt{\mathcal{H}_a}$, which is pivotal for our kernel estimates.

We first define the auxiliary functions that appear in the spectral measure formula. For $s > 0$ and angular variables $\theta, \theta'\in\mathbb{S}^1$, let
\begin{align}
A_{a}(s;\theta,\theta^{\prime}) &:= \big{(}\cos(\sqrt{a}s)-1\big{)} + sE_{1}(s;\theta,\theta^{\prime}) + s^{2}E_{2}(s;\theta,\theta^{\prime}), \label{eq:A_a_def} \\
B_{a}(s;\theta,\theta^{\prime}) &:= \sin(\sqrt{a}\,\pi)e^{-s\sqrt{a}} + D_{1}(s;\theta,\theta^{\prime}) + D_{2}(s;\theta,\theta^{\prime}), \label{eq:B_a_def}
\end{align}
where the components are given by the following series
\begin{align}
E_{1}(s;\theta,\theta^{\prime}) &:= \sum_{k\in\mathbb{Z}\setminus\{0\}}e^{ik(\theta-\theta^{\prime})}\left(-\frac{a}{2|k|}\right)\sin(|k|s), \label{eq:G1_def} \\
E_{2}(s;\theta,\theta^{\prime}) &:= \sum_{k\in\mathbb{Z}\setminus\{0\}}e^{ik(\theta-\theta^{\prime})}O(|k|^{-2}), \label{eq:G2_def} \\
D_{1}(s;\theta,\theta^{\prime}) &:= \sum_{k\in\mathbb{Z}\setminus\{0\}}e^{ik(\theta-\theta^{\prime})}\sin(\sqrt{k^{2}+a}\,\pi)\big{(}e^{-s\sqrt{k^{2}+a}}-e^{-s|k|}\big{)}, \label{eq:D1_def} \\
D_{2}(s;\theta,\theta^{\prime}) &:= \sum_{k\in\mathbb{Z}\setminus\{0\}}e^{ik(\theta-\theta^{\prime})}\big{(}\sin(\sqrt{k^{2}+a}\,\pi)-\sin(|k|\pi)\big{)}e^{-s|k|}. \label{eq:D2_def}
\end{align}

The following proposition provides the explicit form of the spectral measure kernel for $\sqrt{\mathcal{H}_a}$, which was established in prior work \cite{ZZ}.

\begin{proposition}\cite[Proposition 5.1]{ZZ} \label{prop:spect}
Let $\mathcal{H}_{a}$ be the operator given in \eqref{La}. Let $x=(r\cos\theta,r\sin\theta)\in\mathbb{R}^{2}$, $x^{\prime}=(r^{\prime}\cos\theta^{\prime},r^{\prime}\sin\theta^{\prime})\in\mathbb{R}^{2}$ and $y,y^{\prime}\in\mathbb{R}^{n}$. Let $dE_{\sqrt{\mathcal{H}_{a}}}(\lambda;X,X^{\prime})$ be the spectral measure kernel, where $X=(x,y)$ and $X'=(x',y')$. Then
\begin{equation}\label{kernel:sp}
\begin{split}
dE_{\sqrt{\mathcal{H}_{a}}}(\lambda;X,X^{\prime}) = &\frac{\lambda^{n+1}}{\pi^{n+1}}\sum_{\pm}a_{\pm}(\lambda |X-X^{\prime}|)e^{\pm i\lambda|X-X^{\prime}|} \\
&+ \frac{\lambda^{n+1}}{\pi^{n+1}}\int_{|\theta-\theta^{\prime}|}^{\pi}\sum_{\pm}a_{\pm}(\lambda|\vec{n}_{s}^{1}|)e^{\pm i\lambda|\vec{n}_{s}^{1}|} \times A_{a}(s;\theta,\theta^{\prime})\,ds \\
&- \frac{\lambda^{n+1}}{\pi^{n+1}}\int_{0}^{\infty}\sum_{\pm}a_{\pm}(\lambda|\vec{n}_{s}^{2}|)e^{\pm i\lambda|\vec{n}_{s}^{2}|} \times B_{a}(s,\theta,\theta^{\prime})\,ds,
\end{split}
\end{equation}
where the symbols $a_{\pm}(\lambda)$ satisfy the derivative bounds
\begin{equation}
|\partial_{\lambda}^{k}a_{\pm}(\lambda)| \leq C_{k}(1+\lambda)^{-\frac{n+1}{2}-k}, \quad \forall k \geq 0,
\end{equation}
the functions $A_{a}$ and $B_{a}$ are defined in \eqref{eq:A_a_def} and \eqref{eq:B_a_def}, and the vctors $\vec{n}_{s}^{1}$ and $\vec{n}_{s}^{2}$ are given by
\begin{equation}\label{equ:n}
\begin{split}
\vec{n}_{s}^{1} &= \left(r-r^{\prime},\ \sqrt{2rr^{\prime}(1-\cos s)},\ y-y^{\prime}\right), \\
\vec{n}_{s}^{2} &= \left(r+r^{\prime},\ \sqrt{2rr^{\prime}(\cosh s-1)},\ y-y^{\prime}\right).
\end{split}
\end{equation}
\end{proposition}

A crucial ingredient for handling the remainder terms in the spectral measure is the uniform integrability of $A_a$ and $B_a$, which is provided in the following lemma.

\begin{lemma}\cite{ZZ} \label{lem:int_AB}
Let $A_{a}(s;\theta,\theta^{\prime})$ and $B_{a}(s;\theta,\theta^{\prime})$ be defined as in \eqref{eq:A_a_def} and \eqref{eq:B_a_def}. Then there exists a constant $C$, independent of $\theta$ and $\theta'$, such that
\begin{equation}
\int_{0}^{\pi}|A_{a}(s;\theta,\theta^{\prime})|\,ds + \int_{0}^{\infty}|B_{a}(s;\theta,\theta^{\prime})|\,ds \leq C. \label{equ:int_AB}
\end{equation}
\end{lemma}

We will also utilize the Gaussian boundedness for the heat kernel associated with $\mathcal{H}_a$, which is stated in the next proposition.

\begin{proposition}\cite[Proposition 3.1]{ZZ} \label{prop:heat_kernel}
Let $\mathcal{H}_{a}$ be the operator given in \eqref{La}. Suppose $H_{a}(t;X,X^{\prime})$ is the heat kernel of $e^{-t\mathcal{H}_{a}}$. Then it admits the decomposition
\begin{equation}
H_{a}(t;X,X^{\prime}) = H_{0}(t;X,X^{\prime}) + R_{a}(t;X,X^{\prime}), \label{eq:heat_kernel_decomp}
\end{equation}
where the principal part $H_{0}$ satisfies the Gaussian upper bound
\begin{equation}
\left|H_{0}(t;X,X^{\prime})\right| \leq C\,|t|^{-\frac{n+2}{2}} e^{-\frac{|X-X^{\prime}|^{2}}{4t}}, \label{eq:H0_bound}
\end{equation}
and the remainder term $R_{a}$ is controlled by
\begin{equation}
\left|R_{a}(t;X,X^{\prime})\right| \leq C\,|t|^{-\frac{n+2}{2}} e^{-\frac{|X-X^{\prime}|^{2}}{8t}}. \label{eq:Ra_bound}
\end{equation}
Here, $C$ is a positive constant.
\end{proposition}

Finally, as the central object of our study is the $L^p$ estimate of an analytic family of operators, and we define analytic family of operators as follows.

\begin{definition}
Let $z = z_0 + iz_1 \in \mathbb{C}$ with $z_0,z_1 \in \mathbb{R}$. We define the analytic family of operators $\{T_z\}$ by
\begin{equation}\label{eq:T_z_def}
T_{z} := (1 + \sqrt{\mathcal{H}_a})^{-z} e^{i\sqrt{\mathcal{H}_a}}.
\end{equation}
Using the spectral theorem, $T_z$ can be represented as an integral operator
\begin{equation}\label{eq:T_z_integral}
T_{z}f(X) = \int_{\mathbb{R}^{2+n}} K_{z}(X, X') f(X')\, dX',
\end{equation}
with the kernel given by
\begin{equation}\label{eq:K_z_kernel}
K_{z}(X, X') = \int_{0}^{\infty} (1 + \lambda)^{-z} e^{i\lambda}\, dE_{\sqrt{\mathcal{H}_{a}}}(\lambda; X, X').
\end{equation}
\end{definition}

The subsequent analysis will focus on establishing suitable bounds for the kernel $K_z(X, X')$, which will ultimately lead to the proof of our main $L^p$-boundedness result.

\section{Proof of Proposition \ref{prop:tz_bounds}}\label{proof-prop}
We now prove the main proposition of this paper, which is stated as follows.
\begin{proposition} \label{prop:tz_bounds}
Let $T_{z}$ be the analytic family of operators defined by $T_{z} = (1 + \sqrt{\mathcal{H}_a})^{-z}e^{i\sqrt{\mathcal{H}_a}}$. Then, there exists a constant $C > 0$ such that
\begin{align}
\|T_{z}f\|_{L^{2}(\mathbb{R}^{2+n})} &\leq C \|f\|_{L^{2}(\mathbb{R}^{2+n})}, \quad \Re(z) = 0, \label{eq:L2_bound} \\
\|T_{z}f\|_{L^{1}(\mathbb{R}^{2+n})} &\leq C \|f\|_{L^{1}(\mathbb{R}^{2+n})}, \quad \Re(z) = \frac{n+1}{2} + \varepsilon, \quad \varepsilon > 0. \label{eq:L1_bound}
\end{align}
\end{proposition}

\begin{proof}
By the unitarity of the operator $e^{i\sqrt{\mathcal{H}_a}}$, the estimate \eqref{eq:L2_bound} follows immediately. It therefore remains to prove \eqref{eq:L1_bound}.
From the Proposition \ref{prop:spect} and \eqref{eq:K_z_kernel}, the kernel of the operator $T_z$  can be expressed as follows:
\begin{equation}
\begin{split}
K_z(X,X')
&=\int_0^\infty (1+\lambda)^{-z} e^{i\lambda}\Big(\frac{\lambda^{n+1}}{\pi^{n+1}}\sum_{\pm}a_{\pm}(\lambda|(x,y)-(x',y')|)e^{\pm i\lambda|(x,y)-(x',y')|}\nonumber\\
&\qquad+\frac{\lambda^{n+1}}{\pi^{n+1}}\int_{|\theta-\theta'|}^\pi\sum_{\pm}a_{\pm}(\lambda|\vec{n}^1_s|)e^{\pm i\lambda|\vec{n}^1_s|}\times A_a(s; \theta,\theta') \, ds\\
&\qquad-\frac{\lambda^{n+1}}{\pi^{n+1}}\int^\infty_0\sum_{\pm}a_{\pm}(\lambda|\vec{n}^2_s|)e^{\pm i\lambda|\vec{n}^2_s|}\times B_a(s,\theta,\theta')\, ds\Big)d\lambda.
\end{split}
\end{equation}
Then, it is sufficient to verify
\begin{align}\label{est:L1-Kz}
\Big\|\int_{\mathbb{R}^{2+n}}K_z(X,X')f(X')dX'\Big\|_{L^1(\mathbb{R}^{2+n})}
\leq C \|f\|_{L^1(\mathbb{R}^{2+n})}.
\end{align}

Next, fix a bump function
$\beta\in C^\infty_0((1/2,2))$ satisfying
\begin{equation}\label{beta-d}
\sum_{j=-\infty}^\infty \beta(2^{-j} \lambda)=1, \quad \lambda>0,
\end{equation}
and define $\beta(0)=1,$ we set
$$\beta_{0}(\lambda)=\sum_{j\le 0} \beta(2^{-j}\lambda)
\in C^\infty_0([0,2)).$$
Then, the kernel $K_z(X,X')$ is decomposed into the sum of the following three terms, defined as follows:
\begin{align}
&U(X,X')=\int_0^\infty \Big((1+\lambda)^{-z} e^{i\lambda}-1\Big) \beta_0(\lambda)\,  dE_{\sqrt{\mathcal{H}_{{a}}}}(\lambda;X,X');\label{def:low-U}\\
&V(X,X')=\int_0^\infty \beta_0(\lambda) \,  dE_{\sqrt{\mathcal{H}_{{a}}}}(\lambda;X,X');\label{def:low-V}\\
&G(X,X')=\sum_{j\geq1}\int_0^\infty (1+\lambda)^{-z} e^{i\lambda} \beta_j(\lambda)\,  \label{def:high-G} dE_{\sqrt{\mathcal{H}_{{a}}}}(\lambda;X,X').
\end{align}
Applying \eqref{est:L1-Kz}, we are reduce to show
\begin{equation}\label{est:low-U}
\Big\|\int_{\mathbb{R}^{2+n}}U(X,X')f(X')dX'\Big\|_{L^1(\mathbb{R}^{2+n})}\leq C\|f\|_{L^1(\mathbb{R}^{2+n})},
\end{equation}
\begin{equation}\label{est:low-V}
\Big\|\int_{\mathbb{R}^{2+n}}V(X,X')f(X')dX'\Big\|_{L^1(\mathbb{R}^{2+n})}\leq C\|f\|_{L^1(\mathbb{R}^{2+n})},
\end{equation}
and
\begin{equation}\label{est:high-G}
\Big\|\int_{\mathbb{R}^{2+n}}G(X,X')f(X')dX'\Big\|_{L^1(\mathbb{R}^{2+n})}\leq C\|f\|_{L^1(\mathbb{R}^{2+n})}.
\end{equation}
First, we begin by proving the simpler case \eqref{est:low-V}. Let $\psi(x)=\beta_0(\sqrt x)$ and $\psi_e(x):=\psi(x)e^{2x}.$ Then $\psi_e$ is a $C_c^\infty$-function on $\mathbb{R}$ with support in $[0,4).$ Consequently, its Fourier transform $\hat\psi_e$ belongs to the Schwartz class. We write
\begin{align*}
\beta_0(\sqrt x)&=\psi(x)=e^{-2x}\psi_e(x)=e^{-2x}\int_{\mathbb{R}}
e^{ix\cdot\xi}\hat\psi_e(\xi)d\xi\\
&=e^{-x}\int_{\mathbb{R}}e^{-x(1-i\xi)}\hat\psi_e(\xi)d\xi.
\end{align*}
Therefore, via the functional calculus, we obtain
\begin{align*}
\beta_0(\sqrt{\mathcal{H}_{a}})=\psi(\mathcal{H}_{a})=e^{-\mathcal{H}_{a}}\int_{\mathbb{R}}
e^{-(1-i\xi)\mathcal{H}_{a}}\hat\psi_e(\xi)d\xi.
\end{align*}
By the proposition \ref{prop:heat_kernel}, we have the heat kernel estimates
\begin{align}\label{heat-kernel}
\Big|e^{-t\mathcal{H}_{a}}(X,X')\Big|\leq C |t|^{-\frac{2+n}{2}}e^{-\frac{|X-X'|^2}{8t}}.
\end{align}
Based on the heat kernel estimates \eqref{heat-kernel}, we have
\begin{align*}
|V(X,X')|=\big|\beta_0(\sqrt{\mathcal{H}_{a}})(X,X')\big|\leq& C\int_{\mathbb{R}^{2+n}}e^{-\frac{|X-Y|^2}{8}}e^{-\frac{|Y-X'|^2}
{8}}dY\int_{\mathbb{R}}|\hat\psi_e(\xi)|d\xi\\
\leq& C e^{-\frac{|X-X'|^2}{16}}\int_{\mathbb{R}^{2+n}} e^{-\frac{|X-Y|^2}{16}}e^{-\frac{|Y-X'|^2}
{16}}dY\\
\leq& C e^{-\frac{|X-X'|^2}{16}},
\end{align*}
which implies
\begin{align*}
\Big\|\int_{\mathbb{R}^{2+n}}V(X,X')f(X')dX'\Big
\|_{L^1(\mathbb{R}^{2+n})}\leq & C\int_{\mathbb{R}^{2+n}}\int_{\mathbb{R}^{2+n}} e^{-\frac{|X-X'|^2}{16}}|f(X')|dX'dX\\
\leq&C \int_{\mathbb{R}^{2+n}}|f(X')|dX'\int_{\mathbb{R}^{2+n}}
e^{-\frac{|X-X'|^2}{16}}dX\\
\leq &C\|f\|_{L^1(\mathbb{R}^{2+n})}.
\end{align*}
Thus, this yields \eqref{est:low-V}.

Next, we focus on the proof of \eqref{est:low-U}. Let $d=|X-X'|,$ and $U:=\frac{1}{\pi^{n+1}}(U_1+U_2+U_3).$ We consider
\begin{align*}
&U_1(X,X')=\int_0^\infty \Big((1+\lambda)^{-z} e^{i\lambda}-1\Big) \beta_0(\lambda)\lambda^{n+1} a_\pm(\lambda d)e^{\pm i\lambda d}d\lambda,\\
&U_2(X,X')=\int_0^\infty \int_{|\theta-\theta'|}^\pi\Big((1+\lambda)^{-z} e^{i\lambda}-1\Big) \beta_0(\lambda)\lambda^{n+1} a_\pm(\lambda |\vec{n}^1_s|)e^{\pm i\lambda |\vec{n}^1_s|} A_a(s; \theta,\theta') \, ds d\lambda,\\
&U_3(X,X')=\int_0^\infty \int^\infty_0\Big((1+\lambda)^{-z} e^{i\lambda}-1\Big) \beta_0(\lambda)\lambda^{n+1} a_\pm(\lambda |\vec{n}^2_s|)e^{\pm i\lambda |\vec{n}^2_s|} B_a(s,\theta,\theta')\, ds, d\lambda.
\end{align*}
On the one hand, we further decompose
\begin{equation}
\begin{split}
U_1(X,X')&= \int_0^\infty \Big((1+\lambda)^{-z}e^{i\lambda}-1\Big) e^{\pm i\lambda d} a_{\pm}(\lambda d) \beta_0(\lambda)\beta_0(\lambda d)\lambda^{n+1} d\lambda\\
&\quad+\sum_{j\geq 1} \int_0^\infty \Big((1+\lambda)^{-z}e^{i\lambda}-1\Big) e^{\pm i\lambda d} a_{\pm}(\lambda d) \beta_0(\lambda)\beta_j(\lambda d)\lambda^{n+1} d\lambda,
\end{split}
\end{equation}
then one has
\begin{equation}
\begin{split}
&\Big| \int_0^\infty \Big((1+\lambda)^{-z}e^{i\lambda}-1\Big) e^{\pm i\lambda d} a_{\pm}(\lambda d) \beta_0(\lambda)\beta_0(\lambda d)\lambda^{n+1} d\lambda\Big|\\
&\leq C \int_{\lambda<\min\{1,d^{-1}\}} \lambda^{n+2} d\lambda\leq C(1+d)^{-(n+3)}.
\end{split}
\end{equation}
On the other hand, employing integration by parts N-times, in which $N$ is an integer satisfying $N \geq \left\lceil\frac{n+5}{2}\right\rceil + 1$, it yields
\begin{equation}
\begin{split}
&\sum_{j\geq 1}\Big| \int_0^\infty \Big((1+\lambda)^{-z}e^{i\lambda}-1\Big) e^{\pm i\lambda d}  a_{\pm}(\lambda d) \beta_0(\lambda)\beta_j(\lambda d)\lambda^{n+1} d\lambda\Big|\\
&\leq C d^{-N}\sum_{j\geq 1}\int_{2^jd^{-1}\sim \lambda\leq 1}(1+\lambda d)^{-\frac{n+1}{2}} (1+\lambda)^{-\Re z}\lambda^{n+2-N} d\lambda\\
&\leq C d^{-\frac{n+1}{2}-N} \sum_{j\geq 1}\int_{2^j d^{-1}\sim \lambda\leq 1} \lambda^{\frac{n+3}2-N} d\lambda\\
&\leq C(1+d)^{-(n+3)},
\end{split}
\end{equation}
where we use the fact
\begin{align*}
(1+\lambda d)^{-\frac{n+1}{2}} (1+\lambda)^{-\Re z}\leq (\lambda d)^{-\frac{n+1}{2}},
\end{align*}
and
\begin{align*}
\lambda d\sim 2^j,~~d\sim 2^j/\lambda>1,~j\geq 1.
\end{align*}
Therefore, we obtain
$$|U_1(X,X')|\leq C (1+d)^{-(n+3)},$$
which implies
\begin{equation}\label{U1}
\begin{split}
&\Big\|\int_{\R^{2+n}} U_1(X,X') f(X') dX'\Big\|_{L^1(\R^{2+n})}\\
&\quad\leq C\int_{\R^{2+n}}\int_{\R^{2+n}} (1+|X-X'|)^{-(n+3)} |f(X')|dX'dX \leq C\|f\|_{L^1(\mathbb{R}^{2+n})}.
\end{split}
\end{equation}
By analogy with the estimation for $U_1$, the same method is applied to $U_2$ and $U_3$, yielding the following estimates
\begin{equation}
|U_2(X,X')|\leq C \int_{|\theta-\theta'|}^\pi (1+|\vec{n}^1_s|)^{-(n+3)}|A_a(s; \theta,\theta')|ds,
\end{equation}
and
\begin{equation}
|U_3(X,X')|\leq C \int^\infty_0 (1+|\vec{n}^2_s|)^{-(n+3)}|B_a(s; \theta,\theta')|ds.
\end{equation}
Thus, by Lemma \ref{lem:int_AB}, we can get
\begin{equation}\label{U2}
\begin{split}
&\Big\|\int_{\R^{2+n}} U_2(X,X') f(X') dX'\Big\|_{L^1(\R^{2+n})}\\
&\quad\leq C\int_{\R^{2+n}}\int_{\R^{2+n}} \int_{|\theta-\theta'|}^\pi(1+|\vec{n}^1_s|)^{-(n+3)} |A_a(s; \theta,\theta')|ds  |f(X')|dX'dX\\
&\quad \leq  \int_{\R^{2+n}}\int_{|\theta-\theta'|}^\pi\Big(\int_{\R^{2+n}} (1+|\vec{n}^1_s|)^{-(n+3)}dX\Big) |A_a(s; \theta,\theta')|ds |f(X')|dX' \\
&\quad \leq C\|f\|_{L^1(\mathbb{R}^{2+n})},
\end{split}
\end{equation}
and
\begin{equation}\label{U3}
\begin{split}
&\Big\|\int_{\R^{2+n}} U_3(X,X') f(X') dX'\Big\|_{L^1(\R^{2+n})}\\
&\quad\leq C\int_{\R^{2+n}}\int_{\R^{2+n}} \int^\infty_0(1+|\vec{n}^2_s|)^{-(n+3)} |B_a(s; \theta,\theta')|ds|f(X')|dX  dX'\leq C\|f\|_{L^1(\mathbb{R}^{2+n})}.
\end{split}
\end{equation}
In fact, from \eqref{equ:n}, we observe that
\begin{equation*}
\begin{split}
|\vec{n}^1_s|^2= r^2+r'^2-2rr'\cos s+|y-y'|^2\geq |(x-x', y-y')|^2, \, \text{for}\,|\theta-\theta'|\leq s\leq \pi
\end{split}
\end{equation*}
and
\begin{equation*}
\begin{split}
|\vec{n}^2_s|^2\geq (r+r')^2+|y-y'|^2\geq |(x-x', y-y')|^2.
\end{split}
\end{equation*}
Hence, it follows from \eqref{equ:int_AB} that
\begin{align*}
\int_{|\theta-\theta'|}^\pi&\Big(\int_{\R^{2+n}} (1+|\vec{n}^1_s|)^{-(n+3)}dX\Big) |A_a(s; \theta,\theta')|ds\\
&\leq \int_{|\theta-\theta'|}^\pi\Big(\int_{\R^{2+n}} (1+|X-X'|)^{-(n+3)}dX\Big) |A_a(s; \theta,\theta')|ds\leq C,
\end{align*}
Similarly, \eqref{U3} can be obtained. From \eqref{U1},\eqref{U2} and \eqref{U3}, then \eqref{est:low-U} follows.

Finally, we prove \eqref{est:high-G}.
We decompose $G(X,X')$ into the following three distinct terms
\begin{equation*}
G_1(X,X')=\sum_{j\geq1}\int_0^\infty (1+\lambda)^{-z} e^{i\lambda} \beta_j(\lambda)\lambda^{n+1}a_{\pm}(\lambda d)e^{\pm i \lambda d}d\lambda,
\end{equation*}
\begin{equation*}
G_2(X,X')=\sum_{j\geq1}\int_0^\infty \int_{|\theta-\theta'|}^\pi(1+\lambda)^{-z} e^{i\lambda} \beta_j(\lambda)\lambda^{n+1}a_{\pm}(\lambda |\vec{n}^1_s|)e^{\pm i \lambda |\vec{n}^1_s|}A_a(s; \theta,\theta') \, dsd\lambda,
\end{equation*}
and
\begin{equation*}
G_3(X,X')=\sum_{j\geq1}\int_0^\infty\int_0^\infty  (1+\lambda)^{-z} e^{i\lambda} \beta_j(\lambda)\lambda^{n+1}a_{\pm}(\lambda |\vec{n}^2_s|)e^{\pm i \lambda |\vec{n}^2_s|}B_a(s,\theta,\theta')\, dsd\lambda,
\end{equation*}
then one has $G=\frac{1}{\pi^{n+1}}(G_1+G_2+G_3)$. Next we begin by deriving
\begin{equation}\label{est:G_1}
\Big\|\int_{\mathbb{R}^{2+n}}G_1(X,X')f(X')dX'\Big\|_{L^1(\mathbb{R}^{2+n})}\leq C\|f\|_{L^1(\mathbb{R}^{2+n})}.
\end{equation}
For computational convenience, we denote
\begin{equation}
\begin{split}
G_1(X,X')&=\sum_{j\geq1}\int_0^\infty (1+\lambda)^{-z}e^{i\lambda(1\pm d)}  a_{\pm}(\lambda d) \beta(2^{-j}\lambda)\lambda^{n+1} d\lambda\\
&:=\sum_{j\geq1}I_j(X,X').
\end{split}
\end{equation}

To prove \eqref{est:G_1}, we estimate $G_1$ by considering two cases as follows, depending on the distance of $d=|X-X'|$ from 1.

{\bf Case 1: $||X-X'|-1|\geq\delta$ ($0<\delta<\frac 1 2$).} By integration by parts N-times (N large enough), we can establish
\begin{equation}\label{Gh1}
\begin{split}
|G_1(X,X')|&\leq C(1\pm d)^{-N}\sum_{j\geq1}\int_0^\infty (1+\lambda)^{-\Re z}(1+\lambda d)^{-\frac{n+1}{2}}\beta(2^{-j}\lambda)\lambda^{n+1-N} d\lambda\\
&\leq C(1\pm d)^{-N}.
\end{split}
\end{equation}

{\bf Case 2: $||X-X'|-1|<\delta.$} In this case, it implies $|X-X'|\sim1.$ Utilizing integration by parts one-time, one has
\begin{equation}\label{I1}
\begin{split}
|I_j(X,X')|&=\Big|\int_0^\infty (1+\lambda)^{-z}e^{i\lambda(1\pm d)}  a_{\pm}(\lambda d) \beta(2^{-j}\lambda) \lambda^{n+1}d\lambda\Big|\\
&\leq C(1\pm d)^{-1}\int_{\lambda\sim 2^{j}}\lambda^{-1}(1+\lambda)^{-\Re z}(1+\lambda d)^{-\frac{n+1}{2}}\beta(2^{-j}\lambda)\lambda^{n+1} d\lambda\\
&\leq C(1\pm d)^{-1}\int_{\lambda\sim 2^{j}}(1+\lambda)^{-1-\varepsilon}d\lambda\\
&\leq C 2^{-\varepsilon j}(1\pm d)^{-1}.
\end{split}
\end{equation}
Through  a direct calculation, it can be obtained that
\begin{equation}\label{I2}
\begin{split}
|I_j(X,X')|&=\Big|\int_0^\infty (1+\lambda)^{-z}e^{i\lambda(1\pm d)}  a_{\pm}(\lambda d) \beta(2^{-j}\lambda)\lambda^{n+1} d\lambda\Big|\\
&\leq \int_{\lambda\sim 2^{j}}
(1+\lambda)^{-\frac{n+1} 2-\varepsilon}(1+\lambda)^{-\frac {n+1} 2}\lambda^{n+1} d\lambda\leq C 2^{(1-\varepsilon)j}.
\end{split}
\end{equation}
Combining \eqref{I1} and \eqref{I2}, a straightforward calculation yields
\begin{equation}\label{I3}
\begin{split}
I_j(X,X')\leq C\Big(2^{-\varepsilon j}(1\pm|X-X'|)^{-1}\Big)^{1-\iota}
\Big(2^{(1-\varepsilon)j}\Big)^{\iota}
\end{split}
\end{equation}
where we choose $\iota=\frac{\varepsilon}{2}$ and $0<\varepsilon<1$, hence
\begin{equation}\label{Gh2}
\begin{split}
|G_1(X,X')|&= C\sum_{j\geq1}\Big(2^{-\varepsilon j}(1\pm|X-X'|)^{-1}\Big)^{1-\iota}
\Big(2^{(1-\varepsilon)j}\Big)^\iota\\
&\leq C\sum_{j\geq1}2^{(-\varepsilon+\iota)j}
(1\pm|X-X'|)^{-(1-\iota)}\\
&\leq C (1\pm|X-X'|)^{-(1-\iota)}.
\end{split}
\end{equation}
Therefore, by \eqref{Gh1} and \eqref{Gh2}, which implies
\begin{equation}\label{G1}
\begin{split}
\Big\|&\int_{\R^{2+n}} G_1(X,X') f(X') dX'\Big\|_{L^1(\R^{2+n})}\leq \int_{\R^{2+n}}\int_{\R^{2+n}}|G_1(X,X')|dX|f(X')|dX'\\
&\leq\int_{\R^{2+n}}\Big(\int_{||X-X'|-1|\geq\delta}+
\int_{||X-X'|-1|<\delta}\Big)|G_1(X,X')|dX|f(X')|dX'
\\
&\leq\int_{\R^{2+n}}\Big(\int_{||X-X'|-1|\geq\delta}(1\pm |X-X'|)^{-N}dX\\
&\qquad +\int_{||X-X'|-1|<\delta}
(1\pm|X-X'|)^{-(1-\iota)}dX\Big)|f(X')|dX'\\
&\leq C\|f\|_{L^1(\mathbb{R}^{2+n})}.
\end{split}
\end{equation}
So far, we have completed the proof of \eqref{est:G_1}.

Next, similar to $G_1$, we turn our attention to the terms $G_2$ and $G_3$. To simplify the notation, we denote
\begin{equation}\label{G_2}
\begin{split}
G_2(X,X')&=\sum_{k\geq1}\int_0^\infty \int_{|\theta-\theta'|}^\pi(1+\lambda)^{-z} e^{i\lambda(1\pm|\vec{n}^1_s|)}  a_{\pm}(\lambda|\vec{n}^1_s|) \beta(2^{-j}\lambda)\lambda^{n+1} A_a(s; \theta,\theta') \, dsd\lambda\\
&:=\sum_{k\geq1}\int_{|\theta-\theta'|}^\pi J_k(s,X,X')A_a(s; \theta,\theta') \, ds
\end{split}
\end{equation}
where
\begin{equation}
J_k(s,X,X')=\int_0^\infty(1+\lambda)^{-z} e^{i\lambda(1\pm|\vec{n}^1_s|)}  a_{\pm}(\lambda|\vec{n}^1_s|) \beta(2^{-j}\lambda)\lambda^{n+1}d\lambda,
\end{equation}
and we write
\begin{equation}\label{Jk}
J_k(X,X')=\int_{|\theta-\theta'|}^\pi J_k(s,X,X')A_a(s; \theta,\theta') \, ds.
\end{equation}

To estimate $G_2$, we distinguish two cases based on the size of $\big||\vec{n}^1_s|-1\big|.$

{\bf Case 1: $||\vec{n}^1_s|-1|\geq\delta$ ($0<\delta<\frac 1 2$).} Applying integration by parts N-times once more (N large enough), we can obtain
\begin{equation}\label{G2-1}
\begin{split}
&|G_2(X,X')|\\
&\leq C\int_0^\infty \int_{|\theta-\theta'|}^\pi |1\pm|\vec{n}^1_s||^{-N}\sum_{j\geq1}  (1+\lambda)^{-\Re z}(1+\lambda|\vec{n}^1_s|)^{-\frac{n+1}{2}}
\beta(2^{-j}\lambda)\lambda^{n+1-N} |A_a(s; \theta,\theta') |\, dsd\lambda \\
&\leq C\int_{|\theta-\theta'|}^\pi|1\pm|\vec{n}^1_s||^{-N}|A_a(s; \theta,\theta') |\, ds.
\end{split}
\end{equation}

{\bf Case 2: $||\vec{n}^1_s|-1|<\delta $.} It implies $(1+\lambda|\vec{n}^1_s|)^{-1/2}\leq C (1+\lambda)^{-\frac 1 2}.$
A further integration by part yields
\begin{equation}\label{J1}
\begin{split}
&|J_k(s,X,X')|\\
&=\Big|\int_0^\infty (1+\lambda)^{-z} e^{i\lambda(1\pm|\vec{n}^1_s|)}  a_{\pm}(\lambda|\vec{n}^1_s|) \beta(2^{-j}\lambda)\lambda^{n+1} d\lambda\Big|\\
&\leq C|1\pm|\vec{n}^1_s||^{-1}\int_{\lambda\sim 2^{j}}\lambda^{-1}(1+\lambda)^{-\Re z}(1+\lambda|\vec{n}^1_s|)^{-\frac{n+1}{2}}\beta(2^{-j}\lambda)\lambda^{n+1} d\lambda \\
&\leq C|1\pm|\vec{n}^1_s||^{-1}\int_{\lambda\sim 2^{j}}(1+\lambda)^{-\frac{n+1} 2-\varepsilon}(1+\lambda)^{-\frac {n+1} 2}\lambda^n d\lambda \\
&\leq C 2^{-\varepsilon j}|1\pm|\vec{n}^1_s||^{-1}.
\end{split}
\end{equation}
Then, through  a direct calculation, it can be obtained that
\begin{equation}\label{J2}
\begin{split}
|J_k(s,X,X')|&=\Big|\int_0^\infty (1+\lambda)^{-z} e^{i\lambda(1\pm|\vec{n}^1_s|)}  a_{\pm}(\lambda|\vec{n}^1_s|) \beta(2^{-j}\lambda)\lambda^{n+1} d\lambda\Big|\\
&\leq \int_{\lambda\sim 2^{j}}
(1+\lambda)^{-\frac{n+1} 2-\varepsilon}(1+\lambda)^{-\frac {n+1} 2}\lambda^{n+1} d\lambda\\
&\leq C2^{(1-\varepsilon)j}.
\end{split}
\end{equation}
Combining \eqref{G_2}, \eqref{J1} and \eqref{J2}, we have
\begin{equation}\label{J3}
\begin{split}
J_k(s,X,X')&\leq C\Big(2^{-\varepsilon j}|1\pm|\vec{n}^1_s||^{-1}\Big)^{1-\iota}
\Big(2^{(1-\varepsilon)j}\Big)^{\iota}\\
&\leq C2^{(-\varepsilon+\iota)j}
|1\pm|\vec{n}^1_s||^{-(1-\iota)}
\end{split}
\end{equation}
by choosing $\iota=\frac{\varepsilon}{2}$ and $0<\varepsilon<1$.
From \eqref{Jk} and \eqref{J3}, we deduce
\begin{align}\label{G2-2}
|G_2(X,X')|\leq C \int_{|\theta-\theta'|}^\pi|1\pm|\vec{n}^1_s||^{-(1-\iota)}|A_a(s; \theta,\theta') |\, ds.
\end{align}
Therefore, by \eqref{G2-1}, \eqref{G2-2}, Lemma \ref{lem:int_AB} and \eqref{est-n1} in appendix \ref{app}, which implies
\begin{equation}\label{G2}
\begin{split}
\Big\|&\int_{\R^{2+n}} G_2(X,X') f(X') dX'\Big\|_{L^1(\R^{2+n})}\leq \int_{\R^{2+n}}\int_{\R^{2+n}}|G_2(X,X')| dX|f(X')|dX' \\
& \leq\int_{\R^{2+n}}\Big(\int_{||\vec{n}^1_s|-1|\geq\delta}+
\int_{||\vec{n}^1_s|-1|<\delta}\Big)|G_2(X,X')| dX|f(X')|dX'\\
&\leq C\int_{\R^{2+n}}\int_{||\vec{n}^1_s|-1|\geq\delta}\int_{|\theta-\theta'|}^\pi
|1\pm|\vec{n}^1_s||^{-N}|A_a(s; \theta,\theta')| \, dsdX|f(X')|dX'\\
&\qquad +\int_{\R^{2+n}}
\int_{||\vec{n}^1_s|-1|<\delta}\int_{|\theta-\theta'|}^\pi
|1\pm|\vec{n}^1_s||^{-(1-\iota)}|A_a(s; \theta,\theta') | \, dsdX|f(X')|dX'\\
&\leq C\|f\|_{L^1(\mathbb{R}^{2+n})}.
\end{split}
\end{equation}
Finally, repeating the above calculation again, the estimate for $G_3$ can be obtained. We now present the proof in detail.
When $||\vec{n}^2_s|-1|\geq\delta,$
we then have
\begin{equation}\label{G3-1}
\begin{split}
&|G_3(X,X')|\\
&\leq C\int_0^\infty \int_0^\infty |1\pm|\vec{n}^2_s||^{-N}\sum_{j\geq1}  (1+\lambda)^{-\Re z}(1+\lambda|\vec{n}^2_s|)^{-\frac{n+1}{2}}
\beta(2^{-j}\lambda)\lambda^{n+1-N} |B_a(s; \theta,\theta') |\, dsd\lambda \\
&\leq C\int_0^\infty |1\pm|\vec{n}^2_s||^{-N}|B_a(s; \theta,\theta') |\, ds;
\end{split}
\end{equation}
when $||\vec{n}^2_s|-1|<\delta$, one can demonstrate
\begin{equation}\label{G3-2}
\begin{split}
|G_3(X,X')|&\leq C\sum_{j\geq1}\int_0^\infty2^{(-\varepsilon+\iota)j}
|1\pm|\vec{n}^2_s||^{-(1-\iota)}|B_{\alpha}(s,\theta_1,\theta_2)|ds\\
&\leq C \int_0^\infty|1\pm|\vec{n}^2_s||^{-(1-\iota)}|B_{\alpha}(s,\theta_1,\theta_2)|ds.
\end{split}
\end{equation}
Therefore, given \eqref{G3-1} and \eqref{G3-2}, which implies
\begin{equation}\label{G3}
\begin{split}
\Big\|&\int_{\R^{2+n}} G_3(X,X') f(X') dX'\Big\|_{L^1(\R^{2+n})}\leq \int_{\R^{2+n}}\int_{\R^{2+n}}|G_3(X,X')| dX|f(X')|dX' \\
& \leq\int_{\R^{2+n}}\Big(\int_{||\vec{n}^2_s|-1|\geq\delta}+
\int_{||\vec{n}^2_s|-1|<\delta}\Big)|G_3(X,X')| dX|f(X')|dX'\\
&\leq C\int_{\R^{2+n}}\int_{||\vec{n}^2_s|-1|\geq\delta}\int_0^\infty
|1\pm|\vec{n}^2_s||^{-N}|B_a(s; \theta,\theta')| \, dsdX|f(X')|dX'\\
&\qquad +\int_{\R^{2+n}}
\int_{||\vec{n}^2_s|-1|<\delta}\int_0^\infty
|1\pm|\vec{n}^2_s||^{-(1-\iota)}|B_a(s; \theta,\theta') | \, dsdX|f(X')|dX'\\
&\leq C\|f\|_{L^1(\mathbb{R}^{2+n})}.
\end{split}
\end{equation}
In the last inequality, we have used Lemma \ref{lem:int_AB} and the inequality \eqref{est-n2} from appendix \ref{app}.
From \eqref{G1}, \eqref{G2} and \eqref{G3}, this yields \eqref{est:high-G}. We have thus established Proposition \ref{prop:tz_bounds}.
\end{proof}

\section{Proof of Theorem \ref{thm:main}}\label{proof-thm:main}
This section is devoted to the proof of Theorem \ref{thm:main}. Our approach relies on Proposition \ref{prop:tz_bounds}, analytic interpolation, and duality argument. To this end, the key step is to prove that the operator family $T_z = (1 + \sqrt{\mathcal{H}_a})^{-z} e^{i\sqrt{\mathcal{H}_a}}$ is bounded on the appropriate $L^p$ spaces.

\begin{proof}
From Proposition \ref{prop:tz_bounds}, we have the following uniform estimates for the analytic family $ T_z $:
When $\Re(z) = 0,$
  \begin{equation}\label{TzL2}
  \|T_z f\|_{L^2(\mathbb{R}^{2+n})} \leq C \|f\|_{L^2(\mathbb{R}^{2+n})};
  \end{equation}
When $\Re(z) = \frac{n+1}{2} + \varepsilon ,$
  \begin{equation}\label{TzL1}
  \|T_z f\|_{L^1(\mathbb{R}^{2+n})} \leq C \|f\|_{L^1(\mathbb{R}^{2+n})}.
  \end{equation}

By duality argument, the estimate \eqref{TzL1} implies that $T_z$ is also bounded from $L^\infty$ to $L^\infty$ when $\Re(z) = \frac{n+1}{2} + \varepsilon .$

We now apply Stein's complex interpolation theorem (see \cite{S}) and consider the following two interpolation scenarios. First, interpolate between the $L^2 \to L^2$ boundedness at $\Re(z) = 0$ and the $L^1 \to L^1$ boundedness at $\Re(z) = \frac{n+1}{2} + \varepsilon.$ For $z_0 = \theta(\frac{n+1}{2} + \varepsilon)$ with $0 \leq \theta \leq 1,$ the interpolation yields,
\begin{equation}\label{Inter1}
\|T_{z_0}\|_{L^p \to L^p} \leq C \quad \text{for } z_0 = \theta\left(\frac{n+1}{2} + \varepsilon\right),\,\frac{1}{p} =\frac{1 + \theta}{2}.
\end{equation}
Similarly, interpolate between the $L^2 \to L^2 $ boundedness at $\Re(z) = 0$ and the $L^\infty \to L^\infty $ boundedness at $\Re(z) = \frac{n+1}{2} + \varepsilon .$ With the same $z_0 = \theta(\frac{n+1}{2} + \varepsilon),$ we have,
\begin{equation}\label{Inter2}
\|T_{z_0}\|_{L^p \to L^p} \leq C \quad \text{for } z_0 = \theta\left(\frac{n+1}{2} + \varepsilon\right),\,\frac{1}{p} =\frac{1 - \theta}{2}.
\end{equation}
Combining \eqref{Inter1} and \eqref{Inter2}, we find that $T_{z_0}$ is bounded on $ L^p$ for all $p$ satisfying
\begin{equation}
\left| \frac{1}{p} - \frac{1}{2} \right| \leq \frac{z_0}{n+1+2\varepsilon}.
\end{equation}

To obtain the estimate for real $\gamma > 0,$ we employ a scaling argument. Consider the modified operator family
\begin{equation*}
\tilde{T}_z = (1 + \sqrt{\mathcal{H}_a})^{-z} e^{it\sqrt{\mathcal{H}_a}}.
\end{equation*}
Then the same interpolation procedure applied to $ \tilde{T}_z $ yields
\begin{equation*}
\|\tilde{T}_\gamma f\|_{L^p} \leq C(1 + t)^\gamma \|f\|_{L^p}.
\end{equation*}

Since the operators $(1 + \sqrt{\mathcal{H}_a})^{-\gamma} $ and $ (1 + t\sqrt{\mathcal{H}_a})^{-\gamma} $ are equivalent in \( L^p \)-norm up to a constant factor, we conclude
\begin{equation}
\left\| (1 + \sqrt{\mathcal{H}_a})^{-\gamma} e^{it\sqrt{\mathcal{H}_a}} f \right\|_{L^p(\mathbb{R}^{2+n})} \leq C(p,\gamma)(1 + t)^\gamma \|f\|_{L^p(\mathbb{R}^{2+n})}
\end{equation}
for all $p$ satisfying $\left| \frac{1}{p} - \frac{1}{2} \right| \leq\frac{\gamma}{n+1+2\varepsilon}$ with $\varepsilon > 0.$

This completes the proof of Theorem \ref{thm:main}.
\end{proof}\vspace{0.2cm}

\appendix
\section{Integrability of Singular Integral Kernels  }\label{app}

In this appendix, we establish a technical lemma that is used in the proof of Proposition \ref{prop:tz_bounds} in Section \ref{proof-prop}.

\begin{lemma}\label{lem:est}
Let $x=(r\cos\theta,r\sin\theta)\in\mathbb{R}^{2}$, $x^{\prime}=(r^{\prime}\cos\theta^{\prime},r^{\prime}\sin\theta^{\prime})\in\mathbb{R}^{2}$ and $y,y^{\prime}\in\mathbb{R}^{n}$. Denote $X=(x,y)$ and $X'=(x',y')$.
For $0<\iota<1,$ there exists a constant $C>0$ independent of $X'$ such that
\begin{align}\label{est-n1}
\int_{\mathbb{R}^{2+n}}\int_{|\theta-\theta'|}^\pi \mathbf{1}_{\{ |\vec{n}_s^1|<2 \}} \;
\bigl| 1-|\vec{n}_s^1| \bigr|^{-(1-\iota)} |A_a(s;\theta,\theta')| \, ds dX\leq C ,
\end{align}
and
\begin{align}\label{est-n2}
\int_{\mathbb{R}^{2+n}}\int_0^\infty\mathbf{1}_{\{ |\vec{n}_s^2|<2 \}}
|1-|\vec{n}^2_s||^{-(1-\iota)}|B_a(s; \theta,\theta') | \, ds dX\leq C,
\end{align}
where $\vec{n}_{s}^{1}$ and $\vec{n}_{s}^{2}$  are as defined in \eqref{equ:n}
and  $dX = r\,dr\,d\theta\,dy$.
\end{lemma}

\begin{proof}We first prove \eqref{est-n1}.
Fix $X'=(r',\theta',y')$ and exchange the order of integration, set
\begin{align*}
\Phi_1(X') = \int_{|\theta-\theta'|}^{\pi}
\Bigl[ \int_{\{X:|\vec{n}_s^1|<2\}}
\bigl|1-|\vec{n}_s^1|\bigr|^{-(1-\iota)} \, dX \Bigr]
|A_a(s;\theta,\theta')| \, ds ,
\end{align*}
and denote the inner integral by
\begin{align}\label{zeta1}
\zeta_1(s,X')=\int_{\{X:|\vec{n}_s^1|<2\}}
\bigl|1-|\vec{n}_s^1|\bigr|^{-(1-\iota)} \, dX .
\end{align}
This shows that inequality \eqref{est-n1} holds if and only if $\Phi_1(X')\leq C$. Write $u=r-r', t=y-y', \sigma=1-\cos s\in[0,2]$, $dX = r\,dr\,d\theta\,dy=(r'+u)du\,d\theta\,dt$, and
\[
|\vec{n}_s^1| = \sqrt{u^2+|t|^2 + 2(r'+u)r'\sigma}.
\]
Therefore, the condition $|\vec{n}_s^1|<2$ becomes
\begin{align}\label{ineq1-u,t}
u^2+|t|^2 + 2(r'+u)r'\sigma < 4.
\end{align}
We denote the domain $\Omega=\{(u,t):u^2+|t|^2 + 2(r'+u)r'\sigma < 4\},$ then, \eqref{zeta1} can be written as
\begin{equation}\label{zeta1-u}
\begin{split}
\zeta_1(s,X')&=2\pi\int_{\Omega}|1-\sqrt{u^2+|t|^2 + 2(r'+u)r'\sigma}|^{-(1-\iota)}(r'+u)dudt.
\end{split}
\end{equation}
From \eqref{ineq1-u,t} we obtain $u^2+|t|^2 < 4$, so $(u,t)$ lies in a ball of radius $2$ in $\mathbb{R}^{n+1}$.
The singularity $|1-|\vec{n}_s^1||^{-(1-\iota)}$ occurs when $|\vec{n}_s^1|=1$, i.e.
\[
u^2+|t|^2 = 1 - 2(r'+u)r'\sigma.
\]
The right-hand side is bounded, so the singularity lies within a bounded region.

Let $v=u+r'\sigma$ and $Q:=r'^2\sigma(2-\sigma)\in[0,4]$, and then by using the polar coordinate $\rho=\sqrt{v^2+|t|^2}$ with $\rho<2$, we have
\begin{equation*}
\begin{split} \zeta_1(s,X')&\leq C\int_{v^2+|t|^2<4-Q}\Big|v^2+|t|^2-(1-Q)\Big|^{-(1-\iota)} (r'+2)dvdt\\
&\leq C \omega_n\int_{0}^{\infty}\mathbf{1}_{\{\rho<\sqrt{4-Q}\}}
\Big|\rho^2-(1-Q)\Big|^{-(1-\iota)}(r'+2)\rho^nd\rho
\end{split}
\end{equation*}
provided that \eqref{ineq1-u,t} and \eqref{zeta1-u}.
In the first inequality above, we have applied the relation $|1-|\vec{n}_s^1||\geq \frac{1}{4}|1-|\vec{n}_s^1|^2|$ due to the fact that $|\vec{n}_s^1|<2$. For convenience, here we write
$$\zeta_1(s,X')\leq C\cdot RHS$$
with
\begin{equation}\label{zeta1-rho}
\begin{split}RHS = \int_{0}^{\infty}\mathbf{1}_{\{\rho<\sqrt{4-Q}\}}
\Big|\rho^2-(1-Q)\Big|^{-(1-\iota)}(r'+2)\rho^nd\rho.
\end{split}
\end{equation}
Now we reduce to prove the following integral
\begin{equation*}
\begin{split}
\int_{|\theta-\theta'|}^{\pi}RHS|A_a|\,ds
\end{split}
\end{equation*}
is bounded. Due to the presence of the factor $(r'+2)$ in \eqref{zeta1-rho}, a more refined estimate is required. Therefore, we proceed by considering two cases according to the size of $r'$: $r' \le 4$ and $r' > 4$.

{\bf Case 1: $r'\leq4.$}
From \eqref{zeta1-rho}, the integrand has a (integrable) singularity at $\rho = \sqrt{|1-Q|}$.
For $n\geq1$, let $\xi=\rho^2-(1-Q),$ we obtain
\begin{align*}
RHS\leq& C\int_{0}^{\infty}\mathbf{1}_{\{\rho^2<4-Q\}}
\Big|\rho^2-(1-Q)\Big|^{-(1-\iota)}d\rho^2\\
\leq &C\int_{-1}^4|\xi|^{-(1-\iota)}d\xi\leq C,
\end{align*}
where the last inequality follows because the exponent satisfies $-(1-\iota) > -1$ and the integration interval is bounded. Thus, combining Lemma \ref{lem:int_AB} with $s\geq |\theta-\theta'|$, we establish
\[
\Phi_1(X') \leq C\int_{|\theta-\theta'|}^{\pi}RHS|A_a|\,ds\le C \int_0^\pi |A_a|\, ds \le  C .
\]

{\bf Case 2: $r' > 4$.} Condition \eqref{ineq1-u,t} forces $r'^2\sigma(2-\sigma) \leq 4$ for large $r'$, hence
\begin{equation}\label{r'large}
\sigma \leq \frac{4}{r'^2}, \quad s \leq \frac{2}{r'} \qquad \text{or}\qquad 2-\sigma \leq \frac{4}{r'^2}, \quad \pi-s \leq \frac{2}{r'} \quad (\text{exclude})
\end{equation}
However, the second case can be ruled out, because if the inequality $\pi-s \leq \frac{2}{r'}$ hold for sufficiently large $r'$, it would imply $\cos s \le 0$, and therefore $|\vec{n}_s^1|^2 \ge r^2 + r'^2$, which contradicts the given condition $r'<|\vec{n}_s^1| < 2$.

Thus, for large $r'$, the effective interval of $s$ has length of order $1/r'$, and consequently $s \leq \frac{2}{r'} $.
Then using \eqref{r'large} together with \eqref{zeta1-rho}, we attain
\begin{equation*}
\begin{split}
\Phi_1(X')& \leq C \int_0^\pi \int_{0}^{\infty}\mathbf{1}_{\{s\leq \frac{2}{r'}\}}\mathbf{1}_{\{\rho<\sqrt{4-Q}\}}
\Big|\rho^2-(1-Q)\Big|^{-(1-\iota)}\rho^n  \, (r'+2)d\rho |A_a|ds.
\end{split}
\end{equation*}
Combining the expression of $A_a=A_a(s;\theta,\theta')$ given in \eqref{eq:A_a_def} with the bounds for $E_1,E_2$ provided in \eqref{eq:G1_def} and \eqref{eq:G2_def} (for details, please refer to \cite{ZZ}.), it directly follows that $|A_a|\leq Cs$ for large $r'.$ Hence,
\begin{equation}\label{r'ge2}
\begin{split}
\Phi_1(X')& \leq C \int_0^\pi \int_{0}^{\infty}\mathbf{1}_{\{s\leq \frac{2}{r'}\}}\mathbf{1}_{\{\rho<\sqrt{4-Q}\}}
\Big|\rho^2-(1-Q)\Big|^{-(1-\iota)}\rho^n  \, (r'+2)d\rho sds\\
&\leq C \int_0^\pi \int_{0}^{\infty}\mathbf{1}_{\{s\leq \frac{2}{r'}\}}\mathbf{1}_{\{\rho<\sqrt{4-Q}\}}
\Big|\rho^2-(1-Q)\Big|^{-(1-\iota)}\rho^n  \, \frac{(r'+2)}{r'}d\rho ds\leq C.
\end{split}
\end{equation}
Similar to the estimate for \eqref{zeta1-rho} above, we can also establish that \eqref{r'ge2} holds here. Now, we complete the proof of \eqref{est-n1}.\vspace{0.2cm}

We now turn to the estimate \eqref{est-n2}. Because the integrability of the singularity differs between the cases $n=1$ and $n \ge 2$, we handle them separately. For $n=1$, the boundedness of the integral is obtained directly; for $n \ge 2$, a dyadic decomposition in a neighborhood of the singularity is employed to refine the estimate.
First, we recall
\[
\vec n_s^2 = \bigl( r+r',\; \sqrt{2rr'(\cosh s-1)},\; y-y' \bigr).
\]
Hence $|\vec n_s^2|<2$ implies
\begin{align}\label{n2-long}
r^2+r'^2+2rr'\cosh s + |y-y'|^2<4.
\end{align}
Interchanging the order of integration gives
\[
\Phi_2(X') = \int_{0}^{\infty} \zeta_2(s,X')\,
                |B_a(s;\theta,\theta')|\,ds,
\]
with
\[
\zeta_2(s,X') = \int_{\{X:|\vec n_s^2|<2\}}
                |1-|\vec n_s^2||^{-(1-\iota)}\,dX,
\qquad dX = r\,dr\,d\theta\,dy.
\]

{\bf Case 1: $n=1.$} On the one hand, by condition \eqref{n2-long}, we can acquire
\[
r^2+|y-y'|^2 < 4 - r'^2 - 2rr'\cosh s \le 4 - r'^2,
\]
which implies $r,r'<2$ and $|y-y'|<2$. Hence, the integration region is bounded.
On the other hand, the integration exhibits a singularity when $|\vec n_s^2|=1$, i.e. when
\[
r^2+r'^2+2rr'\cosh s + |y-y'|^2 = 1.
\]

Thus, making the change of variables $v=r+r'\cosh s$, $z=y-y'$, and $\rho = v^2+z^2$,
we set $\eta = \rho^2 - \bigl(r'^2\cosh^2 s + 1\bigr)$.
Condition \eqref{n2-long} implies $\eta \in [-1,3)$.
We then estimate $\zeta_2$ as follows:
\begin{align*}
&\zeta_2(s,X')\\
&\leq C\int_{\rho<\sqrt{4+r'^2\cosh^2 s }}\Big|(r+r'\cosh s)^2+|y-y'|^2-(r'^2\cosh^2 s+1)\Big|^{-(1-\iota)}drdy\\
&\leq C\int \Big(\mathbf{1}_{\rho<\sqrt{\frac 1 2+ r'^2\cosh^2 s}}+\mathbf{1}_{\sqrt{2+r'^2\cosh^2 s }<\rho<\sqrt{4+r'^2\cosh^2 s }}\Big)|\rho^2-(r'^2\cosh^2 s+1)|^{-(1-\iota)} \rho d\rho\\
&\qquad +C\int_{\sqrt{\frac 1 2+ r'^2\cosh^2 s}<\rho<\sqrt{2+r'^2\cosh^2 s }}|\rho^2-(r'^2\cosh^2 s+1)|^{-(1-\iota)} \rho d\rho\\
&\leq C\Big(\int_{-\frac 1 2}^1 |\eta|^{-(1-\iota)}d\eta+\int_{-1}^{-\frac 1 2}|\eta|^{-(1-\iota)}d\eta+\int_1^3|\eta|^{-(1-\iota)}d\eta\Big)\leq C.
\end{align*}
We concluded that the three integrals above are bounded by utilizing the fact  $-(1-\iota) > -1$. Consequently, we attain
\[
\zeta_2(s,X') \le C \quad \text{for all } s\ge0 \text{ and all } X' \text{ with } n=1,
\]
where $C$ is independent of $s$ and $X'$.

Finally, employing Lemma \ref{lem:int_AB},
we obtain
\[
\Phi_2(X') \le C \int_{0}^{\infty} |B_a(s;\theta,\theta')|\,ds
           \le C.
\]

{\bf Case 2: $n\ge2.$} Set $g(r,r',s)=r^2+r'^2+2rr'\cosh s,$ then we have $|\vec n_s^2|^2=g(r,r',s)+|y-y'|^2.$
Thus, we divide the region $|\vec{n}_{s}^{2}|<2$  into two cases: $1<|\vec{n}_{s}^{2}|<2$ and $|\vec{n}_{s}^{2}|<1.$
Without loss of generality, we now apply a dyadic decomposition to the first case $1<|\vec{n}_{s}^{2}|<2$.
For $j\geq1,$ let
$$
2^{-j}\leq|\vec{n}_{s}^{2}|-1\leq2^{-j+1},
$$
which implies
\begin{align}\label{ineq-g}
(1+2^{-j})^2-g\leq|y-y'|^2\leq (1+2^{-j+1})^2-g,
\end{align}
where $g(r_1,r_2,s)=r^2+r'^2+2rr'\cosh s\geq0.$

Hence for $0<\iota<1$, by Lemma \ref{lem:int_AB}, we obtain that ($dX=dxdy$)
\begin{align*}
&\int_{\mathbb{R}^{2+n}}\int_{0}^\infty \mathbf{1}_{\{ 1<|\vec{n}_s^2|<2 \}} \;
\bigl| 1-|\vec{n}_s^2| \bigr|^{-(1-\iota)} |B_a(s;\theta,\theta')| \, ds dX\\
&\leq \sum_{j=1}^\infty \int_{0}^\infty\int_{g<(1+2^{-j+1})^2}
\int_{(1+2^{-j})^2-g\leq|y-y'|^2\leq (1+2^{-j+1})^2-g}(2^{-j})^{-(1-\iota)}dydx |B_a|ds\\
&\leq \sum_{j=1}^\infty \int_{0}^\infty\int_{g<(1+2^{-j+1})^2}2^{j(1-\iota)}
\Big([(1+2^{-j+1})^2-g]^{\frac n 2}-[\max\{(1+2^{-j})^2-g,0\}]^{\frac n 2}\Big)dx |B_a|ds\\
&\leq C\sum_{j=1}^\infty \int_{0}^\infty\int_{r<2} 2^{-\iota j}|B_a|dxds\leq C.
\end{align*}
Here we have utilized this inequality
\begin{align}\label{bound-g}
[(1+2^{-j+1})^2-g]^{\frac n 2}-[\max\{(1+2^{-j})^2-g,0\}]^{\frac n 2}\leq C2^{-j}.
\end{align}
In fact, from \eqref{ineq-g}, we have $g<(1+2^{-j+1})^2.$ If additionally $g>(1+2^{-j})^2,$ then $\max\{(1+2^{-j})^2-g,0\}=0$. Applying the mean-value theorem, we gain
$$(1+2^{-j+1})^2-(1+2^{-j})^2= 2(1+\kappa)(2^{-j+1}-2^{-j})\leq C2^{-j}$$
with $2^{-j}<\kappa<2^{-j+1},$  which implies
$$
[(1+2^{-j+1})^2-g]^{\frac n 2}\leq [(1+2^{-j+1})^2-(1+2^{-j})^2]^{\frac n 2}\leq C2^{-\frac n 2 j}\leq C 2^{-j}.
$$
Moreover, if $g<(1+2^{-j})^2,$ then we obtain
\begin{align*}
&[(1+2^{-j+1})^2-g]^{\frac n 2}-[\max\{(1+2^{-j})^2-g,0\}]^{\frac n 2}\\
&=[(1+2^{-j+1})^2-g]^{\frac n 2}-[(1+2^{-j})^2-g]^{\frac n 2}\\
&\leq [(1+2^{-j+1})^2-(1+2^{-j})^2]\cdot n\varsigma[(1+\varsigma)^2-g]^{\frac{n-2}{2}}\qquad 2^{-j}<\varsigma<2^{-j+1}\\
&\leq C2^{-j}.
\end{align*}
Consequently, \eqref{bound-g} follows. For another case $|\vec{n}_{s}^{2}|<1,$ we set $2^{-j}\leq1-|\vec{n}_{s}^{2}|\leq2^{-j+1},$ the proof is analogous.
Now, we have completed the proof of Lemma \ref{lem:est}.
\end{proof}

\subsection*{Acknowledgements}
C. B. Xu was partially supported by National Natural Science Foundation of China (No.12401296) and  Qinghai Natural Science Foundation (No.2024-ZJ-976). J. Zhang was supported by National Natural Science Foundation of China(12531005) and Beijing Natural Science Foundation(1242011).

\begin{center}

\end{center}

\end{document}